\documentclass[12pt]{article}
\usepackage{amsmath}
\usepackage{amssymb}
\usepackage{amsthm}
\usepackage{verbatim}
\usepackage{mathrsfs}

\newcommand{\R}[0]{\mathbb R}

\newcommand{\Ds}[0]{\mathcal D}

\newtheorem{Th}{Theorem}[section]
\newtheorem{Lemma}{Lemma}[section]
\newtheorem{Prop}[Lemma]{Proposition}

\begin{document}

\title{On the local well--posedness of the two component $b$-family of equations}
\author{H. Inci}

\maketitle

\begin{abstract}
	In this paper we consider the two component $b$-family of equations on $\R$. We write the equations on a Sobolev type diffeomorphism group. As an application of this formulation we show that the dependence on the initial data is nowhere locally uniformly continuous. In particular it is nowhere locally Lipschitz and nowhere locally H\"older continuous.
\end{abstract}

\section{Introduction}\label{section_introduction}

The initial value problem for the two component $b$-family of equations on $\R$ is given by
\begin{align}
	\nonumber
	&u_t - u_{txx} + (b+1) u u_x = bu_x u_{xx}+u u_{xxx} + \rho \rho_x,\quad t > 0,\;x \in \R,\\
	\label{b_eq}
	&\rho_t+(\rho u)_x=0,\quad t > 0,\;x \in \R,\\
	\nonumber
	&u(t=0)=u_0,\;\rho(t=0)=\rho_0,
\end{align}
which is a generalization of the $b$-family of equations -- see \cite{bfamily} for the $b$-family of equations corresponding to \eqref{b_eq} in the case $\rho \equiv 0$. For $b=2$ we get the two component Camassa-Holm equation and for $b=3$ the two component Degasperis-Procesi equation.\\
Using Kato's semigroup theory it was shown in \cite{bfamily_lwp} that \eqref{b_eq} is locally well-posed in $(u,\rho) \in H^s(\R) \times H^{s-1}(\R),\;s \geq 2$. Later this was improved to $(u,\rho) \in H^s(\R) \times H^{s-1}(\R),\;s > 3/2$, in \cite{bfamily_lwp}, where local well posedness in a range of Besov spaces was established. As a by-product of our diffeomorphism group formulation we will get the same local well-posedness result in Sobolev spaces as in \cite{bfamily_lwp}, i.e. 

\begin{Th}\label{th_lwp}
Let $s > 3/2$. For every $(u_0,\rho_0) \in H^s(\R) \times H^{s-1}(\R)$ there is a $T > 0$ s.t. there is a unique pair
	\[
		(u,\rho) \in C([0,T];H^s(\R) \times H^{s-1}(\R)) \cap C^1([0,T];H^{s-1}(\R) \times H^{s-2}(\R)),
	\]
satisfying \eqref{b_eq}. For $T > 0$ we denote by $U_T \subset H^s(\R) \times H^{s-1}(\R)$ the set of initial values $(u_0,\rho_0)$ for which the solution to \eqref{b_eq} exists longer than time $T$. Then the time $T$ solution map
	\[
		\Phi_T:U_T \to H^s(\R) \times H^{s-1}(\R),\quad (u_0,\rho_0) \mapsto (u(T),\rho(T))
	\]
is continuous. Here $(u(T),\rho(T))$ is the time $T$ value of the solution $(u,\rho)$ corresponding to the initial value $(u_0,\rho_0)$.
\end{Th}

A natural question is how regular the solution map $\Phi_T$ is, e.g. whether $\Phi_T$ is $C^1$ or at least locally Lipschitz. In \cite{bfamily_nonuniform} it was shown that there is a bounded set in $H^s(\R) \times H^{s-1}(\R),\; s > 5/2$, on which $\Phi_T$ is not uniformly continuous. We will improve this both w.r.t $s$ and w.r.t. to the non uniformity. Our main result reads as

\begin{Th}\label{th_nonuniform}
Let $s > 3/2$ and $T > 0$. Then
	\begin{align*}
		\Phi_T:U_T &\subset H^s(\R) \times H^{s-1}(\R) \to H^s(\R) \times H^{s-1}(\R),\\
		(u_0,\rho_0) &\mapsto \Phi_T((u_0,\rho_0))=(u(T),\rho(T)),
	\end{align*}
is nowhere locally uniformly continuous.
\end{Th}

Theorem \ref{th_nonuniform} tells us that $\Phi_T$ fails to be uniformly continuous on any ball $B \subset U_T$ regardless of how small the ball is. In particular $\Phi_T$ is nowhere locally Lipschitz, nowhere locally H\"older continuous and hence nowhere $C^1$.\\ \\
Our strategy to prove Theorem \ref{th_nonuniform} is similar to the procedure in \cite{bfamily} and consists of two steps. In a first step we will write \eqref{b_eq} in Lagrangian coordinates as an equation on a diffeomorphism group, i.e. we consider the flow map $\varphi$ of $u$
\[
	\varphi_t(t,x)=u(t,\varphi(t,x)),\quad \varphi(0,x)=x,
\]
and write \eqref{b_eq} in terms of $\varphi$. The second equation in \eqref{b_eq} reads as
\[
	\frac{d}{dt} \left(\varphi_x \cdot \rho \circ \varphi\right) =0
\]
or
\[
	\rho(t)=\left(\frac{\rho_0}{\varphi_x}\right) \circ \varphi(t)^{-1}.
\]
In the second step we will use this composite expression for a ``moving hump'' argument to produce non uniformity.

\section{Lagrangian formulation}

The goal of this section is to write the equations \eqref{b_eq} in terms of the flow map of $u$, i.e. in terms of $\varphi$ given by
\[
	\varphi_t(t)=u(t) \circ \varphi(t),\; t \geq 0,\quad \varphi(0)=\text{id}.
\]
Here $\text{id}:\R \to \R,\;x \mapsto x$, is the identity map. We introduced in \cite{composition} the diffeomorphism group $\Ds^s$ based on Sobolev spaces. This space will be the configuration space for $\varphi$. More precisely for $s > 3/2$ we define
\[
	\Ds^s(\R):=\{\varphi:\R \to \R \;|\; \varphi-\text{id} \in H^s(\R),\; \operatorname{det}(d_x \varphi) > 0 \;\forall x \in \R \}.
\]
Here $H^s(\R)$ is the Sobolev space of order $s$, i.e.
\[
	H^s(\R):=\{f \in L^2(\R) \;|\; \|f\|_{H^s}:=\left(\int_\R (1+|\xi|^2)^s |\hat f(\xi)|^2 \;d\xi \right)^{1/2} < \infty \},
\]
where $\hat f$ is the Fourier transform of $f$. By the Sobolev Imbedding Theorem the function space $\Ds^s(\R),\; s > 3/2$, consists of $C^1$ diffeomorphisms. By the imbedding
\[
	\Ds^s(\R) \to H^s(\R),\quad \varphi \mapsto \varphi - \text{id}
\]
we can identify $\Ds^s(\R)$ with an open subset of $H^s(\R)$, thus we get a differential structure on $\Ds^s(\R)$. In \cite{composition} it was shown that the space $\Ds^s(\R),\;s > 3/2$, is a topological group under composition. That $\Ds^s(\R)$ is the right space follows from a result in \cite{lagrangian}, that says that for every $u \in C([0,T];H^s(\R)),\; T > 0$, there is a unique $\varphi \in C^1([0,T];\Ds^s(\R))$ satisfying
\[
	\varphi_t(t)=u(t) \circ \varphi(t),\; t \in [0,T],\quad \varphi(0)=\text{id}.
\]
To get a Lagrangian formulation of \eqref{b_eq} we write the first equation in \eqref{b_eq} in non local form
\begin{equation}\label{b_nonlocal}
	u_t + uu_x = (1-\partial_x^2)^{-1} \left(-b uu_x+(b-3)u_x u_{xx} + \rho \rho_x \right)
\end{equation}
Now let $\varphi$ be the flow map of $u$. By differentiating $\varphi_t=u \circ \varphi$ w.r.t. $t$ we get
\[
	\varphi_{tt}=(u_t+uu_x) \circ \varphi.
\]
Using \eqref{b_nonlocal} we get
\[
	\varphi_{tt} = \left((1-\partial_x^2)^{-1} \left(-b uu_x+(b-3)u_x u_{xx} + \rho \rho_x \right)\right) \circ \varphi.
\]
As noted in the introduction we have
\[
	\rho=\left(\frac{\rho_0}{\varphi_x}\right) \circ \varphi^{-1}.
\]
If we use this expression for $\rho$ and $u=\varphi_t \circ \varphi^{-1}$ we get
\begin{align*}
	\nonumber
	\varphi_{tt} &= \Big((1-\partial_x^2)^{-1} \Big(-b \varphi_t \circ \varphi^{-1} \cdot \partial_x (\varphi_t \circ \varphi^{-1}) +(b-3)\partial_x(\varphi_t \circ \varphi^{-1}) \partial_x^2(\varphi_t \circ \varphi^{-1}) \\
	&+  \left(\frac{\rho_0}{\varphi_x}\right) \circ \varphi^{-1} \cdot \partial_x \left(\left(\frac{\rho_0}{\varphi_x}\right) \circ \varphi^{-1} \right) \Big)\Big) \circ \varphi=:F(\varphi,\varphi_t,\rho_0).
\end{align*}

We've proved in \cite{bfamily} that the expressions appearing in $F(\varphi,\varphi_t,\rho_0)$ are analytic in $(\varphi,\varphi_t,\rho_0)$. More precisely
\begin{align*}
	\Ds^s(\R) \times H^s(\R) &\to H^s(\R),\\
	(\varphi,\varphi_t) &\mapsto \left((1-\partial_x^2)^{-1} \left( \varphi_t \circ \varphi^{-1} \cdot \partial_x (\varphi_t \circ \varphi^{-1})\right)\right) \circ \varphi,
\end{align*}
and
\begin{align*}
	\Ds^s(\R) \times H^s(\R) &\to H^s(\R),\\
	(\varphi,\varphi_t) &\mapsto \left((1-\partial_x^2)^{-1} \left(\partial_x(\varphi_t \circ \varphi^{-1}) \partial_x^2(\varphi_t \circ \varphi^{-1}) \right)\right) \circ \varphi,
\end{align*}
and
\begin{align*}
	\Ds^s(\R) \times H^{s-1}(\R) &\to H^s(\R),\\
	(\varphi,\rho_0) &\mapsto \left((1-\partial_x^2)^{-1} \left(\left( \frac{\rho_0}{\varphi_x}\right) \circ \varphi^{-1} \cdot \partial_x \left(\left(\frac{\rho_0}{\varphi_x}\right)\circ \varphi^{-1}\right) \right)\right) \circ \varphi,
\end{align*}
are analytic maps. Consult \cite{bfamily} for detailed computations. Let us just point out here that the mechanism underlying analyticity is that conjugation with $\varphi^{-1}$ in combination with $\partial_x$
\[
	 \partial_x (w \circ \varphi^{-1}) \circ \varphi = \frac{w_x}{\varphi_x}
\]
gives an expression involving only derivatives and multiplication. So we get

\begin{Prop}\label{prop_analytic}
Let $s > 3/2$. Then the map
	\[
		\Ds^s(\R) \times H^s(\R) \times H^{s-1}(\R) \to H^s(\R),\quad (\varphi,\varphi_t,\rho_0) \mapsto F(\varphi,\varphi_t,\rho_0)
	\]
is analytic.
\end{Prop}

By applying the Picard-Lindel\"of Theorem we get local in time existence for solutions of \eqref{b_eq}.

\begin{Lemma}\label{local_existence}
Let $s > 3/2$. Then for every initial value $(u_0,\rho_0) \in H^s(\R) \times H^{s-1}(\R)$ there is $T > 0$ and
	\[
		(u,\rho) \in C([0,T];H^s(\R) \times H^{s-1}(\R)) \cap C^1([0,T];H^{s-1}(\R) \times H^{s-2}(\R)),
	\]
satisfying
	\begin{align}
		\nonumber
		&u_t + uu_x = (1-\partial_x^2)^{-1} \left(-b uu_x+(b-3)u_x u_{xx} + \rho \rho_x \right),\\
		\label{nonlocal_eq}
		&\rho_t +(\rho u)_x=0,\\
		\nonumber
		&u(0)=u_0,\;\rho(0)=\rho_0
	\end{align}
	on $[0,T]$.
\end{Lemma}

\begin{proof}
Consider the analytic second order ODE on $\Ds^s(\R)$
	\begin{equation}\label{second_order}
	\varphi_{tt}=F(\varphi,\varphi_t,\rho_0),\; \varphi(0)=\text{id},\; \varphi_t(0)=u_0.
	\end{equation}
	By Picard-Lindel\"of there is $T > 0$ and a solution $\varphi \in \Ds^s(\R)$ to \eqref{second_order} on $[0,T]$. For $t \in [0,T]$ we define
\[
	u(t):=\varphi_t(t) \circ \varphi(t)^{-1},\quad \rho(t):=\left(\frac{\rho_0}{\varphi_x(t)}\right) \circ \varphi(t)^{-1}.
\]
	From the regularity properties of composition established in \cite{composition} we know
\[
		(u,\rho) \in C([0,T];H^s(\R) \times H^{s-1}(\R)) \cap C^1([0,T];H^{s-1}(\R) \times H^{s-2}(\R))
\]
	and from the derivation above we see that $(u,\rho)$ solves \eqref{nonlocal_eq}.
\end{proof}

By the Picard-Lindel\"of Theorem we also get uniqueness.

\begin{Lemma}\label{uniqueness}
	Let $s > 3/2$ and $T > 0$. Suppose that 
	\[
		(u,\rho),(\tilde u,\tilde \rho) \in C([0,T];H^s(\R) \times H^{s-1}(\R)) \cap C^1([0,T];H^{s-1}(\R) \times H^{s-2}(\R))
	\]
	are two solutions to \eqref{nonlocal_eq}. Then $(u,\rho)=(\tilde u,\tilde \rho)$ on $[0,T]$.
\end{Lemma}

\begin{proof}
We know by \cite{lagrangian} that there is a unique $\varphi \in C^1([0,T];\Ds^s(\R))$ s.t.
	\[
		\varphi_t(t)=u(t) \circ \varphi(t),\; t \in [0,T],\quad \varphi(0)=\text{id}.
		\]
Taking the derivative in $\varphi_t=u \circ \varphi$ w.r.t. $t$ we get pointwise the identity
	\[
		\varphi_{tt} = (u_t + uu_x) \circ \varphi.
	\]
By using \eqref{nonlocal_eq} we get pointwise
	\[
		\varphi_{tt} = (1-\partial_x^2)^{-1} \left(-b uu_x+(b-3)u_x u_{xx} + \rho \rho_x \right) \circ \varphi.
	\]
	The right hand side is in $C([0,T];H^s(\R))$. This means that $\varphi \in C^2([0,T];\Ds^s(\R))$ and it solves \eqref{second_order}. The solution $(\tilde u,\tilde \rho)$ generates in a similar fashion a $\tilde \varphi \in C^2([0,T];\Ds^s(\R))$ solving \eqref{second_order}. By Picard-Lindel\"of we get $\varphi=\tilde \varphi$ on $[0,T]$. Therefore we have $u=\tilde u$ on $[0,T]$. This proves uniqueness.
\end{proof}

Combining Lemma \ref{local_existence} and \ref{uniqueness} we can prove Theorem \ref{th_lwp}.

\begin{proof}[Proof of Theorem \ref{th_lwp}]
	By solving \eqref{second_order} as in Lemma \ref{local_existence} we get for $T > 0$ an open set $U_T \subset H^s(\R) \times H^{s-1}(\R)$ as in the statement of the theorem. The solution map is given by
	\begin{align*}
		\Phi_T:U_T &\to H^s(\R) \times H^{s-1}(\R),\\
		(u_0,\rho_0) &\mapsto (\varphi_t(T) \circ \varphi(T)^{-1},\left(\frac{\rho_0}{\varphi_x(T)}\right) \circ \varphi(T)^{-1}),
	\end{align*}
	where $\varphi=\varphi(\cdot;u_0,\rho_0) \in C^2([0,T];\Ds^s(\R))$ is the solution of \eqref{second_order}. We get from the regularity of the composition that this is continuous. Together with the uniqueness of Lemma \ref{uniqueness} this finishes the proof.
\end{proof}

\section{Nonuniform dependence}

In this section we will prove Theorem \ref{th_nonuniform}. Throughout this section we assume $s > 3/2$. Note that \eqref{b_eq} admits the scale invariance
\[
	u_\lambda(t,x):=\lambda u(\lambda t,x),\quad \rho_\lambda(t,x):=\lambda \rho(\lambda t,x),\quad \lambda > 0,
\]
in the sense that $(u_\lambda,\rho_\lambda)$ is a solution to \eqref{b_eq} whenever $(u,\rho)$ is. This scaling shows that $U_T \subset H^s(\R) \times H^{s-1}(\R)$ is star shaped w.r.t. $(0,0) \in H^s(\R) \times H^{s-1}(\R)$. Let us denote $U=\left. U_T \right|_{T=1}$ and $\Phi=\left.\Phi_T\right|_{T=1}$. The scaling $(u_\lambda,\rho_\lambda)$ implies for $T > 0$ and $(u_0,\rho_0) \in U_T$ that $(T u_0,T \rho_0) \in U$ and for the solution map that
\begin{equation}\label{solution_map}
	\Phi_T((u_0,\rho_0))=(u(T),\rho(T))=(\frac{1}{T} u_T(1),\frac{1}{T}\rho_T(1))=\frac{1}{T} \Phi(T u_0,T \rho_0)
\end{equation}
for $(u_0,\rho_0) \in U_T$. Using \eqref{solution_map} Theorem \eqref{th_nonuniform} will follow from 
\begin{Prop}\label{prop_nonuniform}
The map
	\[
		\Phi:U \subset H^s(\R) \times H^{s-1}(\R) \to H^s(\R) \times H^{s-1}(\R)
	\]
is nowhere uniformly continuous.
\end{Prop}

Let us also introduce the time $T=1$ solution map in Lagrangian coordinates, i.e.
\[
	\Psi:U \subset H^s(\R) \times H^{s-1}(\R) \to \Ds^s(\R), (u_0,\rho_0) \mapsto \varphi(1;u_0,\rho_0),
\]
where $\varphi(1;u_0,\rho_0)$ is the time $T=1$ value of the solution to \eqref{second_order} with initial values $(u_0,\rho_0)$. We know by analytic dependence on the initial data that $\Psi$ is an analytic map. Moreover a simple computation shows that
\begin{equation}\label{exp}
	\varphi(t;u_0,\rho_0)=\Psi(t(u_0,\rho_0)).
\end{equation}
Before we prove Proposition \ref{prop_nonuniform} we need the following technical lemma about the map $\Psi$

\begin{Lemma}\label{lemma_dense}
	There is a dense subset $S \subset U$ consisting of smooth compactly supported $(u_\bullet,\rho_\bullet)$ s.t. for every $(u_\bullet,\rho_\bullet) \in S$ there is $w_\bullet=(w_1,0) \in H^s(\R) \times H^{s-1}(\R)$ and $a_\bullet \in \R$ with $\operatorname{dist}(a_\bullet,\operatorname{supp}\rho_\bullet) \geq 2$ satisfying
\[
	\left(d_{(u_\bullet,\rho_\bullet)} \Psi(w_\bullet)\right)(a_\bullet) \neq 0.
\]
	Here $\operatorname{supp} \rho_\bullet \subset \R$ is the support of $\rho_\bullet$, $\operatorname{dist}(a_\bullet,\operatorname{supp}\rho_\bullet)$ is the distance of $a_\bullet$ to the support of $\rho_\bullet$ and $d_{(u_\bullet,\rho_\bullet)} \Psi$ is the differential of $\Psi$ at $(u_\bullet,\rho_\bullet)$, i.e
\begin{align*}
	&d_{(u_\bullet,\rho_\bullet)} \Psi:H^s(\R) \times H^{s-1}(\R) \to H^s(\R),\\
	&v=(v_1,v_2) \mapsto d_{(u_0,\rho_0)} \Psi(v)=\lim_{t \to 0} \frac{\Psi(u_\bullet+tv_1,\rho_\bullet+tv_2)-\Psi(u_\bullet,\rho_\bullet)}{t}.
\end{align*}
\end{Lemma}

\begin{proof}
	For $v=(v_1,v_2) \in H^s(\R) \times H^{s-1}(\R)$ we get from \eqref{exp}  
\[
	d_{(0,0)}\Psi(v)=\left.\frac{d}{dt}\right|_{t=0} \Psi(tv)=\left.\frac{d}{dt}\right|_{t=0} \varphi(t;v_1,v_2)=v_1.
\]
	We know that $C_c^\infty(\R) \times C_c^\infty(\R)$ is dense in $U$. Let $(u_\bullet,\rho_\bullet) \in C_c^\infty(\R) \times C_c^\infty(\R)$ be arbitrary. We take $a_\bullet \in \R$ with $\operatorname{dist}(a_\bullet,\operatorname{supp}\rho_\bullet) \geq 2$ and choose $w_\bullet=(w_1,0) \in H^s(\R) \times H^{s-1}(\R)$ s.t. $w_1(a_\bullet) \neq 0$. Now consider the analytic function
\[
	[0,1] \to \R,\quad t \mapsto \left(d_{t(u_\bullet,\rho_\bullet)}\Psi(w_\bullet)\right)(a_\bullet),
\]
	which for $t=0$ is equal to $w_1(a_\bullet) \neq 0$. Since the function is analytic this means that there is a sequence $t_n \uparrow 1$ s.t.
\[
\left(d_{t_n(u_\bullet,\rho_\bullet)}\Psi(w_\bullet)\right)(a_\bullet) \neq 0,\quad n \geq 1.
\]
So we can put $\{t_n (u_\bullet,\rho_\bullet) \;|\; n \geq 1\}$ into $S$. By doing this for all $(u_\bullet,\rho_\bullet) \in C_c^\infty(\R) \times C_c^\infty(\R)$ we get a dense $S \subset U$ with the desired properties.
\end{proof}

Now we can prove Proposition \ref{prop_nonuniform}.

\begin{proof}[Proof of Proposition \ref{prop_nonuniform}]
	Let $S \subset U$ be as in Lemma \ref{lemma_dense} and $(u_\bullet,\rho_\bullet) \in S$ with corresponding $w_\bullet=(w_1,0) \in H^s(\R) \times H^{s-1}(\R)$ and $a_\bullet \in \R$ satisfying $\operatorname{dist}(a_\bullet,\operatorname{supp}\rho_\bullet)\geq 2$ and $\left(d_{(u_\bullet,\rho_\bullet)}\Psi(w_\bullet)\right)(a_\bullet) \neq 0$. We fix $m > 0$ s.t.
	\[
		|\left(d_{(u_\bullet,\rho_\bullet)}\Psi(w_\bullet)\right)(a_\bullet)| > m \|w_1\|_{H^s}.
	\]
We will determine in succesive steps $R_\ast > 0$ with $B_{R_\ast}((u_\bullet,\rho_\bullet)) \subset U$, where
	\begin{align*}
		&B_{R_\ast}((u_\bullet,\rho_\bullet))=\\
		&\{(u_0,\rho_0) \in H^s(\R) \times H^{s-1}(\R) \;|\; \max\left\{\|u_0-u_\bullet\|_{H^s},\|\rho_0-\rho_\bullet\|_{H^{s-1}}\right\} < R_\ast \},
	\end{align*}
	s.t. $\left. \Phi \right|_{B_R((u_\bullet,\rho_\bullet))}$ is not uniformly continuous for all $0 < R \leq R_\ast$.\\
	The solution map $\Phi=(\Phi_1,\Phi_2)$ has two components, a $u$ and a $\rho$ part. It is sufficient to prove the non uniform dependence for the $\rho$ part
\[
	\Phi_2(u_0,\rho_0)=\left(\frac{\rho_0}{\partial_x \Psi(u_0,\rho_0)}\right) \circ \Psi(u_0,\rho_0)^{-1},\;(u_0,\rho_0) \in U.
\]
	We first choose $R_1 > 0$ with $B_{R_1}((u_\bullet,\rho_\bullet)) \subset U$ and
	\begin{equation}\label{R1}
		\|\Psi(u_0,\rho_0)-\text{id}\|_{H^s},\|\Psi(u_0,\rho_0)^{-1}-\text{id}\|_{H^s}< C_1 
	\end{equation}
for all $(u_0,\rho_0) \in B_{R_1}((u_\bullet,\rho_\bullet))$ and for some $C_1 > 0$. This is clearly possible due to the continuity of $\Psi$. The map
	\[
		H^{s-1}(\R) \times \Ds^s(\R) \to H^{s-1}(\R),\quad (f,\varphi) \mapsto \left(\frac{f}{\varphi_x}\right) \circ \varphi^{-1}
	\]
is continuous as we know from \cite{composition}. By the uniform boundedness principle there is $0 < R_2 \leq R_1$ and $C_2 > 0$ s.t.
	\begin{equation}\label{R2}
		\frac{1}{C_2} \|f\|_{H^{s-1}} \leq \left\| \left(\frac{f}{\partial_x \Psi(u_0,\rho_0)}\right) \circ \Psi(u_0,\rho_0)^{-1} \right\|_{H^{s-1}} \leq C_2 \|f\|_{H^{s-1}}, 
	\end{equation}
for all $f \in H^{s-1}(\R)$ and for all $(u_0,\rho_0) \in B_{R_2}((u_\bullet,\rho_\bullet))$. By Taylor's Theorem we have
	\begin{align*}
		&\Psi_{(u_\bullet+v,\rho_\bullet+w)}=\\
		&\Psi_{(u_\bullet,\rho_\bullet)}+d_{(u_\bullet,\rho_\bullet)}\Psi(v,w)+\int_0^1 (1-s) d_{(u_\bullet+s v,\rho_\bullet+s w)}^2 \Psi \left((v,w),(v,w)\right) \;ds.
	\end{align*}
We will need estimates for the second order differential $d^2\Psi$. Since $\Psi$ is smooth there is $0 < R_3 \leq R_2$ and $C_3 > 0$ s.t.
	\begin{equation}\label{R3a}
		\|d^2_{(u_0,\rho_0)} \Psi((v,w),(\tilde v,\tilde w))\|_{H^s} \leq C_3 (\|v\|_{H^s}+\|w\|_{H^{s-1}})(\|\tilde v\|_{H^s}+\|\tilde w\|_{H^{s-1}})
	\end{equation}
and
	\begin{align}
		\nonumber
		&\|d^2_{(u_0,\rho_0)} \Psi(v,w)^2-d^2_{(\tilde u_0,\tilde \rho_0)}\Psi(v,w)^2\|_{H^s}\\
		\label{R3b}
		&\leq C_3 (\|u_0-\tilde u_0\|_{H^s}+\|\rho_0-\tilde \rho_0\|_{H^{s-1}}) (\|v\|_{H^s}+\|w\|_{H^{s-1}})^2
	\end{align}
	for all $(u_0,\rho_0),(\tilde u_0,\tilde \rho_0) \in B_{R_3}((u_\bullet,\rho_\bullet))$ and for all $(v,w),(\tilde v,\tilde w) \in H^s(\R) \times H^{s-1}(\R)$. Here we used $(u,v)^2=\left((u,v),(u,v)\right)$. By the smoothness of $\Psi$ there is $0 < R_4 \leq R_3$ and $C_4 > 0$ s.t.
	\begin{equation}\label{local_lipschitz}
		\|\Psi(u_0,\rho_0)-\Psi(\tilde u_0,\tilde \rho_0)\|_{H^s} \leq C_4 (\|u_0-\tilde u_0\|_{H^s}+\|\rho_0-\tilde \rho_0\|_{H^{s-1}})
	\end{equation}
for all $(u_0,\rho_0),(\tilde u_0,\tilde \rho_0) \in B_{R_4}((u_\bullet,\rho_\bullet))$. We also fix a constant $C_5 > 0$ for the Sobolev imbedding
	\begin{equation}\label{sobolev}
		\|f\|_{C^1} \leq C_5 \|f\|_{H^s},\quad \forall f \in H^s(\R).
	\end{equation}
After all this choices we set $0 < R_\ast \leq R_5$ in such a way that we have
	\begin{equation}\label{R_ast}
		\max\{C_3 C_5 R_\ast^2/16,C_3 C_5 R_\ast/2\} < m/4.
	\end{equation}
By \eqref{R1} and the Sobolev imbedding \eqref{sobolev} there is $L > 0$ s.t.
	\begin{equation}\label{lipschitz}
		\frac{1}{L}|x-y| \leq	|\Psi(u_0,\rho_0)(x)-\Psi(u_0,\rho_0)(y)| \leq L |x-y|,\quad \forall x,y \in \R,
\end{equation}
	for all $(u_0,\rho_0) \in B_{R_\ast}((u_\bullet,\rho_\bullet))$.\\
Let $0 < R \leq R_\ast$. Our strategy is to construct two sequences of initial data $(u_0^{(n)},\rho_0^{(n)})_{n \geq 1},(\tilde u_0^{(n)},\tilde \rho_0^{(n)})_{n \geq 1} \subset B_R((u_\bullet,\rho_\bullet))$ s.t.
\[
	\lim_{n \to \infty} (\|u_0^{(n)}-\tilde u_0^{(n)}\|_{H^s}+\|\rho_0^{(n)}-\tilde \rho_0^{(n)}\|_{H^{s-1}})=0
\]
whereas
	\[
		\limsup_{n \to \infty} \|\Phi_2(u_0^{(n)},\rho_0^{(n)})-\Phi_2(\tilde u_0^{(n)},\tilde \rho_0^{(n)})\|_{H^{s-1}} > 0.
	\]
This would show that $\left. \Phi \right|_{B_R((u_\bullet,\rho_\bullet))}$ is not uniformly continuous.\\
We define 
\[
	r_n:=\frac{m}{8n} \|w_1\|_{H^s}
\]
and a sequence $(\rho_n)_{n \geq 1} \subset C_c^\infty(\R)$ with $\|\rho_n\|_{H^{s-1}}=R/4$ and
\[
	\operatorname{supp}\rho_n \subset [a_\bullet-\frac{1}{L} r_n,a_\bullet+\frac{1}{L} r_n].
\]
With this we define the pair of initial data as
\[
	u_0^{(n)}=u_\bullet,\;\rho_0^{(n)}=\rho_\bullet+\rho_n,\quad \tilde u_0^{(n)}=u_\bullet+\frac{1}{n} w_1,\;\tilde \rho_0^{(n)}=\rho_\bullet+\rho_n.
\]
There is $N \geq 1$ s.t. 
\[
	(u_0^{(n)},\rho_0^{(n)}),(\tilde u_0^{(n)},\tilde \rho_0^{(n)}) \in B_R((u_\bullet,\rho_\bullet)),\quad \forall n \geq N,
\]
and $\operatorname{supp}\rho_n \subset [a_\bullet-1,a_\bullet+1]$ for $n \geq N$. By construction we have
\[
	\lim_{n \to \infty} (\|u_0^{(n)}-\tilde u_0^{(n)}\|_{H^s}+\|\rho_0^{(n)}-\tilde \rho_0^{(n)}\|_{H^{s-1}})=0.
\]
In order to make the notation easier we introduce for $n \geq N$
\[
	\varphi_n=\Psi(u_0^{(n)},\rho_0^{(n)}),\quad \tilde \varphi_n=\Psi(\tilde u_0^{(n)},\tilde \rho_0^{(n)}).
\]
Thus we get for $n \geq N$
\[
	\Phi_2(u_0^{(n)},\rho_0^{(n)})=\left(\frac{\rho_0^{(n)}}{\partial_x \varphi_n}\right) \circ \varphi_n^{-1},\quad
	\Phi_2(\tilde u_0^{(n)},\tilde \rho_0^{(n)})=\left(\frac{\tilde \rho_0^{(n)}}{\partial_x \tilde \varphi_n}\right) \circ \tilde \varphi_n^{-1}.
\]
For the supports we get
\[
	\operatorname{supp}\left(\frac{\rho_0^{(n)}}{\partial_x \varphi_n}\right) \circ \varphi_n^{-1}=\varphi_n(\operatorname{supp} \rho_\bullet) \cup \varphi_n(\operatorname{supp}\rho_n)
\]
and
\[
	\operatorname{supp}\left(\frac{\tilde \rho_0^{(n)}}{\partial_x \varphi_n}\right) \circ \tilde \varphi_n^{-1}=\tilde \varphi_n(\operatorname{supp} \rho_\bullet) \cup \tilde \varphi_n(\operatorname{supp}\rho_n).
\]
Both are disjoint unions. But we can say more. We introduce the sets
\[
	A_n=\varphi_n(\operatorname{supp} \rho_\bullet),\;B_n=\varphi_n(\operatorname{supp}\rho_n),\;C_n=\tilde \varphi_n(\operatorname{supp} \rho_\bullet),\;D_n=\tilde \varphi_n(\operatorname{supp}\rho_n),
\]
for $n \geq N$. By \eqref{lipschitz} the sets $A_n$ and $B_n$ are separated by a distance not less than $\delta:=1/L > 0$. As 
\[
	\|u_0^{(n)}-\tilde u_0^{(n)}\|_{H^s} + \|\rho_0^{(n)}-\tilde \rho_0^{(n)}\|_{H^{s-1}} \overset{n \to \infty}{\longrightarrow} 0
\]
we get by \eqref{local_lipschitz} and \eqref{sobolev} that for $n$ large enough $A_n \cup C_n$ and $B_n \cup D_n$ are separated by at least $\delta /2$. Thus there is $C_\delta > 0$ s.t.
\begin{align*}
	&\|\Phi_2(u_0^{(n)},\rho_0^{(n)})-\Phi_2(\tilde u_0^{(n)},\tilde \rho_0^{(n)})\|_{H^{s-1}}  = \\
	&\left\| \left(\frac{\rho_0^{(n)}}{\partial_x \varphi_n}\right) \circ \varphi_n^{-1}-\left(\frac{\tilde \rho_0^{(n)}}{\partial_x \tilde \varphi_n}\right) \circ \tilde \varphi_n^{-1}\right\|_{H^{s-1}} \geq \\
	&C_\delta \left(\left\| \left(\frac{\rho_\bullet}{\partial_x \varphi_n}\right) \circ \varphi_n^{-1}-\left(\frac{\rho_\bullet}{\partial_x \tilde \varphi_n}\right) \circ \tilde \varphi_n^{-1}\right\|_{H^{s-1}} \right)+\\
	&C_\delta \left(\left\| \left(\frac{\rho_n}{\partial_x \varphi_n}\right) \circ \varphi_n^{-1}-\left(\frac{\rho_n}{\partial_x \tilde \varphi_n}\right) \circ \tilde \varphi_n^{-1}\right\|_{H^{s-1}} \right) \geq\\
	&C_\delta \left(\left\| \left(\frac{\rho_n}{\partial_x \varphi_n}\right) \circ \varphi_n^{-1}-\left(\frac{\rho_n}{\partial_x \tilde \varphi_n}\right) \circ \tilde \varphi_n^{-1}\right\|_{H^{s-1}} \right) 
\end{align*}
for $n$ large. So it is sufficient to show
\[
	\limsup_{n \to \infty} \left(\left\| \left(\frac{\rho_n}{\partial_x \varphi_n}\right) \circ \varphi_n^{-1}-\left(\frac{\rho_n}{\partial_x \tilde \varphi_n}\right) \circ \tilde \varphi_n^{-1}\right\|_{H^{s-1}} \right) > 0.
\]
We will get this by showing that $B_n$ and $D_n$ are disjoint as well. For that we look at the ``center'' of $B_n$ resp. $D_n$, more precisely at $\varphi_n(a_\ast)$ resp. $\tilde \varphi_n(a_\ast)$. By Taylor's Theorem we have for $\varphi_n=\Psi(u_\bullet,\rho_\bullet + \rho_n)$
\[
	\varphi_n = \Psi(u_\bullet,\rho_\bullet)+d_{(u_\bullet,\rho_\bullet)}\Psi(0,\rho_n)+\int_0^1 (1-s) d_{(u_\bullet,\rho_\bullet + s \rho_n)}^2\Psi (0,\rho_n)^2 \;ds
\]
and similarly for $\tilde \varphi_n=\Psi(u_\bullet+\frac{1}{n}w_1,\rho_\bullet+\rho_n)$
\begin{align*}
	\tilde \varphi_n&=\Psi(u_\bullet,\rho_\bullet)+d_{(u_\bullet,\rho_\bullet)}\Psi(\frac{1}{n} w_1,\rho_n) + \int_0^1 (1-s) d_{(u_\bullet+s \frac{1}{n}w_1,\rho_\bullet + s \rho_n)}^2\Psi (\frac{1}{n}w_1,\rho_n)^2 \;ds\\
&=\Psi(u_\bullet,\rho_\bullet)+d_{(u_\bullet,\rho_\bullet)}\Psi(\frac{1}{n} w_1,\rho_n) + \int_0^1 (1-s) d_{(u_\bullet+s \frac{1}{n}w_1,\rho_\bullet + s \rho_n)}^2\Psi (0,\rho_n)^2 \;ds\\
	&+2 \int_0^1 (1-s) d_{(u_\bullet+s \frac{1}{n}w_1,\rho_\bullet + s \rho_n)}^2\Psi \left((\frac{1}{n}w_1,0),(0,\rho_n)\right) \;ds\\ 
	&+ \int_0^1 (1-s) d_{(u_\bullet+s \frac{1}{n}w_1,\rho_\bullet + s \rho_n)}^2\Psi (\frac{1}{n}w_1,0)^2 \;ds.
\end{align*}
Thus
\begin{align*}
	\tilde \varphi_n- \varphi_n&= d_{(u_\bullet,\rho_\bullet)}\Psi(\frac{1}{n} w_1,0)\\
	&+\int_0^1 (1-s) \left(d_{(u_\bullet+s \frac{1}{n}w_1,\rho_\bullet + s \rho_n)}^2\Psi-d_{(u_\bullet,\rho_\bullet + s \rho_n)}^2\Psi\right) (0,\rho_n)^2 \;ds\\
	&+2 \int_0^1 (1-s) d_{(u_\bullet+s \frac{1}{n}w_1,\rho_\bullet + s \rho_n)}^2\Psi \left((\frac{1}{n}w_1,0),(0,\rho_n)\right) \;ds\\ 
	&+ \int_0^1 (1-s) d_{(u_\bullet+s \frac{1}{n}w_1,\rho_\bullet + s \rho_n)}^2\Psi (\frac{1}{n}w_1,0)^2 \;ds.\\
	&=d_{(u_\bullet,\rho_\bullet)}\Psi(\frac{1}{n} w_1,0)+\mathcal R_1 +\mathcal R_2 + \mathcal R_3.
\end{align*}
Using \eqref{R3b} and \eqref{sobolev} we get
\[
	\|\mathcal R_1\|_\infty \leq C_5 \|\mathcal R_1\|_{H^s} \leq C_5 C_3 \|\frac{1}{n} w_1\|_{H^s} \|\rho_n\|_{H^{s-1}}^2=C_3 C_5 \frac{R^2}{16} \frac{1}{n} \|w_1\|_{H^s}. 
\]
Using \eqref{R3a} and \eqref{sobolev} we get
\[
	\|\mathcal R_2\|_\infty \leq C_5 \|\mathcal R_2\|_{H^s} \leq 2 C_5 C_3 \|\frac{1}{n}w_1\|_{H^s} \|\rho_n\|_{H^{s-1}} = 2 C_3 C_5 \frac{R}{4} \frac{1}{n} \|w_1\|_{H^s}
\]
and
\[
	\|\mathcal R_3\|_\infty \leq C_5 \|\mathcal R_3\|_{H^s} \leq C_5 C_3 \|\frac{1}{n} w_1\|_{H^s}^2= C_3 C_5 \frac{1}{n^2} \|w_1\|_{H^s}^2.
\]
By \eqref{R_ast} we get for $n$ large
\[
	\|\mathcal R_1\|_\infty + \|\mathcal R_2\|_\infty + \|\mathcal R_3\|_\infty < \frac{m}{2n} \|w_1\|_{H^s}.
\]
Therefore 
\begin{align*}
	&|\tilde \varphi_n(a_\bullet)-\varphi_n(a_\bullet)| \geq \left|\left(d_{(u_\bullet,\rho_\bullet)}\Psi(\frac{1}{n}w_1,0)\right)(a_\bullet)\right|-\|\mathcal R_1\|_\infty - \|\mathcal R_2\|_\infty - \|\mathcal R_3\|_\infty \\
	&\geq \frac{1}{n} |\left(d_{(u_\bullet,\rho_\bullet)}\Psi(w_\bullet)\right)(a_\bullet)| -\frac{m}{2n} \|w_1\|_{H^s} = \frac{m}{2n} \|w_1\|_{H^s}
\end{align*}
for $n$ large. Since $\operatorname{supp}\rho_n \subset [a_\bullet-\frac{1}{L}r_n,a_\bullet+\frac{1}{L} r_n]$ we get by \eqref{lipschitz} 
\[
	\operatorname{supp}\left(\frac{\rho_0^{(n)}}{\partial_x \varphi_n}\right) \circ \varphi_n^{-1} \subset [\varphi_n(a_\bullet)-r_n,\varphi_n(a_\bullet)+r_n]
\]
and
\[
	\operatorname{supp}\left(\frac{\rho_0^{(n)}}{\partial_x \tilde \varphi_n}\right) \circ \tilde \varphi_n^{-1} \subset [\tilde \varphi_n(a_\bullet)-r_n,\tilde \varphi_n(a_\bullet)+r_n].
\]
By our choice $r_n=\frac{m}{8n}\|w_1\|_{H^s}$ and the fact the the centers of the intervals have a distance of at least $\frac{m}{2n}\|w_1\|_{H^{s-1}}$, the supports are in such a way apart that we can separate the expressions as we did in \cite{bfamily}. To be precise there is $\tilde C > 0$ s.t.
\begin{align*}
	&\left\| \left(\frac{\rho_n}{\partial_x \varphi_n}\right) \circ \varphi_n^{-1}-\left(\frac{\rho_n}{\partial_x \tilde \varphi_n}\right) \circ \tilde \varphi_n^{-1}\right\|_{H^{s-1}} \\
	&\geq \tilde C \left( \left\| \left(\frac{\rho_n}{\partial_x \varphi_n}\right) \circ \varphi_n^{-1}\right\|_{H^{s-1}} +\left\| \left(\frac{\rho_n}{\partial_x \tilde \varphi_n}\right) \circ \tilde \varphi_n^{-1}\right\|_{H^{s-1}}\right)
\end{align*}
for large $n$. Using this and \eqref{R2} we can estimate
\begin{align*}
	&\limsup_{n \to \infty} \left(\left\| \left(\frac{\rho_n}{\partial_x \varphi_n}\right) \circ \varphi_n^{-1}-\left(\frac{\rho_n}{\partial_x \tilde \varphi_n}\right) \circ \tilde \varphi_n^{-1}\right\|_{H^{s-1}} \right) \\
	&\geq \frac{2 \tilde C}{C_2} \|\rho_n\|_{H^{s-1}}=\frac{2 \tilde C}{C_2}R/4 > 0.
\end{align*}
In summary, we constructed for an arbitrary $R \in (0,R_\ast]$ a pair of sequences 
\[
	(u_0^{(n)},\rho_0^{(n)}),(\tilde u_0^{(n)},\tilde \rho_0^{(n)}) \in B_R((u_\bullet,\rho_\bullet))
\]
with 
\[
	\lim_{n \to 0} \left(\|u_0^{(n)}-\tilde u_0^{(n)}\|_{H^s} + \|\rho_0^{(n)}-\tilde \rho_0^{(n)}\|_{H^{s-1}}\right)=0
\]
whereas
\[
	\limsup_{n \to \infty} \|\Phi(u_0^{(n)},\rho_0^{(n)})-\Phi(\tilde u_0^{(n)},\tilde \rho_0^{(n)})\|_{H^s \times H^{s-1}} > 0.
\]
This shows that $\left.\Phi\right|_{B_R((u_\bullet,\rho_\bullet))}$ is not uniformly continuous. This finishes the proof.
\end{proof}

Using Proposition \ref{prop_nonuniform} we can proof Theorem \ref{th_nonuniform}.

\begin{proof}[Proof of Theorem \ref{th_nonuniform}]
	By \eqref{solution_map} We have for $T > 0$
	\[
		\Phi_T:U_T \to H^s(\R) \times H^{s-1}(\R),\quad (u_0,\rho_0) \mapsto \Phi_T(u_0,\rho_0)=\frac{1}{T} \Phi(T u_0,T \rho_0).
	\]
Thus Proposition \ref{prop_nonuniform} shows that $\Phi_T$ is nowhere uniformly continuous.
\end{proof}

{\bf Acknowledgement.} This work was supported by the BAGEP Award of the Science Academy.

\bibliographystyle{plain}

\flushleft
\author{ Hasan \.{I}nci\\
Department of Mathematics, Ko\c{c} University\\
Rumelifeneri Yolu\\
34450 Sar{\i}yer \.{I}stanbul T\"urkiye\\
        {\it email: } {hinci@ku.edu.tr}
}

\end{document}